\documentclass[11pt,reqno]{amsart}
\usepackage{mathrsfs}
\usepackage{comment}
\usepackage{amsfonts, amsmath,amssymb, amscd, amsthm, bm, cancel}
\usepackage{amsrefs}
\usepackage[
linktocpage=true,colorlinks,citecolor=magenta,linkcolor=blue,urlcolor=magenta]{hyperref}
\newtheorem{proposition}{Proposition}[section]
\newtheorem{theorem}[proposition]{Theorem}
\newtheorem{lemma}[proposition]{Lemma}
\newtheorem{corollary}[proposition]{Corollary}
\newtheorem{definition}[proposition]{Definition}

\newtheorem{question}[proposition]{Question}
\theoremstyle{definition}
\newtheorem*{ack}{Acknowledgements}
\theoremstyle{remark}
\newtheorem{remark}[proposition]{Remark}

\numberwithin{equation}{section}
\allowdisplaybreaks

\newcommand{\Div}{\mathrm{div}}

\begin{document}
\title[Asymptotic Behaviour of WIAMCF]{Asymptotic Behaviour of the weak inverse anisotropic mean curvature flow}

\address{School of Mathematical Sciences, University of Science and Technology of China, Hefei 230026, P.R. China}
\author[C. Gao]{Chaoqun Gao}
\email{\href{mailto:gaochaoqun@mail.ustc.edu.cn}{gaochaoqun@mail.ustc.edu.cn}}
\author[Y. Wei]{Yong Wei}
\email{\href{mailto:yongwei@ustc.edu.cn}{yongwei@ustc.edu.cn}}
\author[R. Zhou]{Rong Zhou}
\email{\href{mailto:zhourong@mail.ustc.edu.cn}{zhourong@mail.ustc.edu.cn}}

\date{\today}

\begin{abstract}
We first establish a local gradient estimate for anisotropic $p$-harmonic functions. A key feature of our estimate is that the constant remains bounded as $p\to 1$; consequently, in the limit $p\to 1$, this estimate yields the local gradient estimate for weak solutions of the inverse anisotropic mean curvature flow (IAMCF). As an application, we show that the weak IAMCF is asymptotic to the expanding Wulff shape solution at the infinity, thereby extending the result of Huisken and Ilmanen in \cite{HI01} to the anisotropic case.
\end{abstract}

\maketitle
\tableofcontents


\section{Introduction}\label{sec:intro}
The inverse mean curvature flow (IMCF) in the Euclidean space $\mathbb{R}^n$ is the evolution of hypersurfaces according to the equation 
\begin{equation}\label{s1.imcf}
    \frac{\partial X}{\partial t}=\frac{1}{H}\nu,
\end{equation}
where $H,\nu$ denote the mean curvature and outer unit normal of the evolving hypersurface $M_t$. Gerhardt \cite{Ge90} and Urbas \cite{Ur90} proved that if the initial hypersurface is smooth, mean convex and star-shaped, then the solution $M_t$ of \eqref{s1.imcf} expands to the infinity and the rescaled hypersurface $e^{-\frac{t}{n-1}}M_t$ converges to a round sphere as $t\to\infty$. The result was improved by Scheuer \cite{Sc16} later, and he showed that the solution is asymptotically round in the sense that the circumradius minus inradius of the evolving hypersurface decays to zero and that the flow becomes close to the expanding spheres.  

Huisken and Ilmanen \cite{HI01} developed the weak solution of IMCF in order to prove the Riemannian Penrose inequality. Let $\Omega \subset \mathbb{R}^{n}~ (n\geq2)$ be an open bounded subset with smooth boundary. The weak inverse mean curvature flow starting from $\Omega$ is described by a locally Lipschitz function $u: \mathbb{R}^{n} \to \mathbb{R}$ satisfying the degenerate elliptic equation
\begin{equation}\label{eq:WIMCF}
    \mathrm{div}\left( \dfrac{Du}{|Du|} \right) = |Du| \quad \text{in} \ \mathbb{R}^{n} \setminus \Omega,
\end{equation}
in certain variational sense, with $\{ u \leq 0 \} = \Omega$ and $u(\infty)=\infty$. The existence of the weak solution was proved in using the elliptic regularization method. An alternative approach was discovered by Moser \cite{M07}, and the idea is that  the weak solution $u$ of IMCF can also be interpreted as the limit when  $p\to 1^+$ of $u_p=(1-p)\log v_p$, where $v_p$ is the $p$-capacity potential of $\Omega$ and so $u_p$ satisfies
\begin{equation}\label{eq:p-WIMCF}
    \begin{cases}
        \Div(|Du_{p}|^{p-2}Du_{p}) = |Du_{p}|^p & \text{in} \ \mathbb{R}^{n} \setminus \overline{\Omega}, \\
        u_{p} = 0 & \text{on} \ \partial\Omega, \\
        u_{p} \to \infty & \text{as} \ |x| \to \infty.
    \end{cases}
\end{equation}
This idea was extended subsequently to IMCF in general Riemannian manifolds \cite{KN09,MRS22}. 

The asymptotical roundness of the weak solution $u$ in asymptotically flat manifolds was shown in \cite[Lemma 7.1]{HI01} using a weak blow down argument. This was extended by Benatti, Fogagnolo, and Mazzieri \cite{BFM24} to weak solution of IMCF in asymptotically conical manifolds. The result in the Euclidean space states as follows:
\begin{theorem}[{\cite[Theorem 1.2]{BFM24} and \cite[Lemma 7.1]{HI01}}]\label{thm:isotropic-asymptotic}
    The weak inverse mean curvature flow starting from an open bounded domain $\Omega\subset\mathbb{R}^{n}$ with smooth boundary satisfies
    \begin{equation}
        u(x) = (n-1)\log |x| -\log \frac{|\partial\Omega^*|}{|\mathbb{S}^{n-1}|}+ o(1) \quad  \text{as} \quad |x| \to \infty,
    \end{equation}
    where $\Omega^*$ denotes the strictly outward minimizing hull of $\Omega$. Moreover, the expanding sphere solution is the only solution to \eqref{eq:WIMCF} on $\mathbb{R}^n\backslash\{0\}$ with compact level sets.
\end{theorem}
The proofs in \cite{BFM24} and \cite{HI01} both rely on the gradient estimate $|Du| \leq C / d(x,0)$ (a consequence of \cite[Theorem 3.1]{HI01}) and then proceed to estimate the eccentricity of the level sets of $u$ for large $t$. The asymptotical roundness of the weak solution is used crucially in estimating the limit of the Hawking mass in order to prove the Penrose inequality. 

We remark that the inverse mean curvature flow in hyperbolic manifolds behaves quite differently. Neves \cite{Ne09} constructed an inverse mean curvature flow in the asymptotically hyperbolic space with positive  mass such that the shape of the limiting hypersurface is not round, and thus it's impossible to prove the hyperbolic Penrose inequailty using the method of IMCF. Later, Hung and Wang \cite{HW15} also showed that the limit shape of the inverse mean curvature flow in the hyperbolic space $\mathbb{H}^n$ is not necessarily round. 

In this paper, we extend the result in Theorem \ref{thm:isotropic-asymptotic} to the inverse anisotropic  mean curvature flow (IAMCF) in the Euclidean space $\mathbb{R}^n$. 

Recall that a function $F \in C^{\infty}(\mathbb{R}^{n} \setminus \{0\})$ is called a \textit{Minkowski norm}, if it defines a norm on $\mathbb{R}^{n}$ and satisfies the uniform ellipticity condition: $D^2(\frac{1}{2}F^2)$ is positive definite on $\mathbb{R}^{n} \setminus \{0\}$. Let $\Omega\subset \mathbb{R}^n$ be an open bounded domain with smooth boundary. The weak solution of IAMCF starting from $\Omega$ with respect to the anisotropic function $F$ is a locally Lipschitz function $u$ satisfying the following degenerate elliptic equation:
\begin{equation}\label{eq:WIAMCF}
	\begin{cases}
		\operatorname{div}(DF(D u)) = F(D u) & \text{in} \ \mathbb{R}^{n} \setminus \Omega, \\
		u = 0 & \text{on} \ \partial\Omega, \\
		u \to \infty & \text{as} \ |x| \to \infty,
	\end{cases}
\end{equation}
in certain variational sense. By using Moser's approximating scheme, Della Pietra, Gavitone and Xia \cite{DGX23} proved the existence of the weak solution of \eqref{eq:WIAMCF}. This was further extended by Cabezas-Rivas, Moll and Solera \cite{CMS24} to IAMCF for general anisotropic function $F$ which includes the crystalline case. The idea is to consider the following approximating equation involving the anisotropic $p$-Laplacian
\begin{equation}\label{eq:p-WIAMCF}
	\begin{cases}
		\Delta_{p,F}u_{p} = F^{p}(D u_{p}) & \text{in} \ \mathbb{R}^{n} \setminus \Omega, \\
		u_{p} = 0 & \text{on} \ \partial\Omega, \\
		u_{p} \to \infty & \text{as} \ |x| \to \infty,
	\end{cases}
\end{equation}
where 
\begin{equation*}
	\Delta_{p,F}u := \operatorname{div}\left(F^{p-1}(D u)DF(D u)\right).
\end{equation*}
By deriving the global gradient estimate of $u_p$, the weak solution of IAMCF is obtained as the limit of $u_p$ as $p\to 1^+$.

\begin{theorem}[\cites{CMS24,DGX23}] Let $\Omega\subset \mathbb{R}^n$ be an open bounded domain with smooth boundary. Then there exists a unique weak solution $u$ of IAMCF \eqref{eq:WIAMCF}.  If there exist two constants $0<r_{1}<r_{2}$ such that $\mathcal{W}_{r_{1}}\subset\Omega\subset\mathcal{W}_{r_{2}}$, then 
		\begin{equation}\label{C0 estimates}
			u_{r_{2}}\leq u \leq u_{r_{1}},
		\end{equation}
        where
        \begin{equation*}
			u_{r_{i}}=(n-1)\log\left(F^{\circ}\left(\frac{\cdot}{r_{i}}\right)\right), \ i=1,2.
        \end{equation*}
    Here $F^\circ$ denotes the dual norm of $F$ and $\mathcal{W}_r=\{x\in \mathbb{R}^n: ~F^\circ(x)<r\}$ denotes the Wulff shape of radius $r$. 
\end{theorem}

\subsection{Local gradient estimate}
As the first result of our paper, we derive a local gradient estimate for the anisotropic $p$-harmonic functions.  Let
\begin{equation*}
    \mathcal{W}_R(x_0)=\{x\in\mathbb{R}^{n}:F^\circ(x-x_0)< R\}
\end{equation*}
denote the Wulff shape centered at $x_0\in\mathbb{R}^n$ with radius $R$ (see \S \ref{sec.2-1} for details). Assume that $v$ is a positive anisotropic $p$-harmonic function in the Wulff shape $\mathcal{W}_R(x_0)$, i.e. $\Delta_{p,F}v=0$ in $\mathcal{W}_R(x_0)$. It's direct to see that $u=(1-p)\log v$ satisfies 
\begin{equation}\label{s1.eq1}
    \Delta_{p,F}u=F^p(Du),\quad \text{in}~ \mathcal{W}_R(x_0).
\end{equation}
We prove the following local gradient estimate for $u$. 
\begin{theorem}\label{Sharp-C1}
	Let $1<p\leq 2$. Assume that $v$ is a positive anisotropic $p$-harmonic function on the Wulff shape $\mathcal{W}_R(x_0)$. Then $u=(1-p)\log v$ satisfies
	\begin{equation}
		\label{equ:sharpgradient}
		\max\limits_{\mathcal{W}_{R/2}(x_0)} F(Du) \leq \dfrac{C(n,F)}{R},
	\end{equation}
	where $C=C(n,F)$ is uniform for $1<p\leq 2$. 
	
\end{theorem}

Our proof of \eqref{equ:sharpgradient} uses the technique of Kotschwar and Ni \cite{KN09} for the local gradient estimate for positive $p$-harmonic functions on general Riemannian manifolds with sectional curvature bounded from below, which is inspired by the gradient estimate of Cheng-Yau \cite{CY75} for positive harmonic functions. Let $f=\frac{1}{2}F^2(Du)$. We first derive a Bochner type formula for $\mathscr{L}f$, where $\mathscr{L}$ is the linearized operator of the nonlinear equation \eqref{s1.eq1}. To estimate the local upper bound of $f$, we multiply $f$ by a suitable cut-off function $\eta\in C_c^\infty(\mathcal{W}_R(x_0))$, derive a differential inequality satisfied by $Q=f\eta$ at its interior maximum point and then apply the maximum principle. The computation is much involved, and a difficulty arises in handling the term containing the third derivative of the Minkowski norm squared $G(\xi):=\frac{1}{2}F^2(\xi)$, which vanishes in the isotropic case. To overcome this, we employ a key property: there exists a constant $C_1 = C_1(F) > 0$ depending only on $F$ such that for all vectors $\alpha, \beta, \gamma \in \mathbb{R}^n$, the following inequality holds:
\begin{equation*}
F D^3 G|_{\xi}(\alpha, \beta, \gamma) \leq C_1 \left( D^2 G|_{\xi}(\alpha, \alpha) \cdot D^2G|_{\xi}(\beta, \beta) \cdot D^2G|_{\xi}(\gamma, \gamma) \right)^{1/2}.
\end{equation*}
This inequality is valid because both sides are homogeneous of degree zero and $D^2G$ is positive definite. Using this estimate, we are able to control the term involving the third derivative of $G$. We note that a similar property, under certain smallness constraints on $C_1 > 0$, has also been used in the interior gradient estimate for anisotropic mean curvature flow in \cite{CL07}.

A remarkable feature of Kotschwar-Ni's gradient estimate \cite{KN09} for positive $p$-harmonic functions is that the constant in their estimate does not blow up when $p \to 1$, which effectively leads to the gradient estimate for the weak solution of IMCF. Recently, Mari, Rigoli, and Setti \cite{MRS22} extended the gradient estimate of Kotschwar and Ni \cite{KN09} to the case where $\mathrm{Ric}\geq -(n-1)h(r)$, with $h(r)$ being a radial function satisfying certain conditions. Our gradient estimate \eqref{equ:sharpgradient} for positive anisotropic $p$-harmonic functions is also uniform for $p$, and thus we obtain the following corollary. 
\begin{corollary}\label{s1.cor1}
	Let $\Omega \subset \mathbb{R}^n$ be an open bounded domain with smooth boundary. Assume that the origin lies in the interior of $\Omega$.  Then there exists a positive constant $C = C(n, F, \Omega)>0$ such that the weak solution $u$ of the IAMCF \eqref{eq:WIAMCF} starting from $\Omega$ satisfies
	\begin{equation}\label{s1.eq:sharpgrad}
		F(Du)(x) \leq \dfrac{C}{F^\circ(x)}
	\end{equation} 
	for almost every $x \in \mathbb{R}^n \setminus \Omega$.
\end{corollary}

We remark that, in addition to the maximum principle method used by Cheng-Yau \cite{CY75} and Kotschwar-Ni \cite{KN09},  the Moser iteration technique is also a powerful tool in deriving gradient estimates. Wang and Zhang \cite{WZ11} employed the Moser iteration to derive a local gradient estimate for positive $p$-harmonic functions on manifolds $M$ satisfying $\mathrm{Ric} \geq -(n-1)\kappa^2$; however, the constant in their estimate becomes unbounded as $p \to 1$. This was further extended by Xia \cite{XQ22} to local gradient estimates for Finsler $p$-harmonic functions on Finsler manifolds with weighted Ricci curvature bounded below, where the constant again blows up as $p \to 1^+$. As a special case of Xia's result, the following local gradient estimate holds for anisotropic $p$-harmonic functions on $\mathbb{R}^n$:
\begin{equation}
		\label{equ:sharpgradient-xia}
		\max\limits_{\mathcal{W}_{R/2}(x_0)} F(Du) \leq \dfrac{C(n,p,F)}{R},
	\end{equation}
However, compared with our estimate in \eqref{equ:sharpgradient}, the constant $C(n,p,F)$ in \eqref{equ:sharpgradient-xia} becomes unbounded as $p\to 1$. For this reason, in Theorem \ref{Sharp-C1} we only focus on the range $1<p\leq 2$, with the main purpose being the application to the weak IAMCF as $p\to 1$.   


\subsection{Asymptotical behaviour of weak IAMCF}
Utilizing the gradient estimate \eqref{s1.eq:sharpgrad} and adapting the methods of \cite[Lemma 7.1]{HI01} and \cite[Theorem 1.2]{BFM24}, we prove the following anisotropic counterpart of Theorem \ref{thm:isotropic-asymptotic}:

\begin{theorem}\label{thm:anisotropic-asymptotic}
	Let $u$ be the weak solution of IAMCF \eqref{eq:WIAMCF} starting from an open bounded  set $\Omega\subset\mathbb{R}^n$ with smooth boundary $\partial\Omega$. Then 
	\begin{equation}\label{s1.ulim}
		u(x) = (n-1) \log F^{\circ}(x)+\log\left( \dfrac{|\partial\mathcal{W}|_F}{|\partial\Omega^*|_F} \right) + o(1) \quad \text{as} \ F^\circ(x) \to \infty,
	\end{equation}
    where $\Omega^*$ is the strictly outward $F$-minimizing hull of $\Omega$. Moreover, the expanding Wulff shape is the only solution to \eqref{eq:WIAMCF} on $\mathbb{R}^n\backslash\{0\}$ with compact level sets.
\end{theorem}

\begin{remark}
The asymptotical behavior of the solution $u_p$ to the equation \eqref{eq:p-WIAMCF} was studied  by Xia-Yin \cite{XY22}. However, since $u_p$ converges to the weak IAMCF $u$ only locally uniformly as $p\to 1^+$, we cannot directly deduce the asymptotical behavior \eqref{s1.ulim} of $u$ by simply letting $p\to 1^+$ in the asymptotical behavior of $u_p$, due to the inherent non-commutativity of limits. Despite this, it is shown a posteriori that the two are indeed consistent.    
\end{remark}

For weak IMCF \eqref{eq:WIMCF}, Huisken-Ilmanen \cite{HI08} proved that the level set $\partial \{u<t\}$ becomes smooth after a sufficiently large time. This combined with Theorem \ref{thm:isotropic-asymptotic} implies the asymptotical behavior of higher order derivatives of $u$. It would be natural to ask the following question.

\begin{question}
Whether the weak IAMCF \eqref{eq:WIAMCF} becomes smooth after a sufficiently large time?
\end{question}



The paper is organized as follows. In Section \ref{sec:prelim}, we recall some preliminaries for anisotropic geometry and weak inverse anisotropic mean curvature flow. In Section \ref{sec:gradientest}, we prove the local gradient estimate for positive anisotropic $p$-harmonic functions, as well as for weak solutions of IAMCF. In Section \ref{sec:main}, we complete the proof of Theorem \ref{thm:anisotropic-asymptotic}, the asymptotic behaviour of the weak IAMCF.

\begin{ack}
The research is supported by National Key Research and Development Program of China 2021YFA1001800, and National Natural Science Foundation of China No.12531002. 
\end{ack}

\section{Preliminaries}\label{sec:prelim}
This section provides the necessary preliminaries for anisotropic geometry and weak solution of the inverse anisotropic mean curvature flow.

\subsection{Minkowski norm and Wulff shape}\label{sec.2-1}

A function $F\in C^{\infty}\left(\mathbb{R}^{n}\backslash\{0\}\right)$ is called a \textit{Minkowski norm} if 
\begin{enumerate}
    \item $F$ is a norm in $\mathbb{R}^{n}$, i.e., $F$ is convex, even, $1$-homogeneous with $F(\xi)\geq 0$ for all $\xi\in \mathbb{R}^n$, and $F(\xi)=0$ if and only if $\xi=0$;
    \item $F$ satisfies a uniformly elliptic condition: $D^2\left( \frac{1}{2}F^2 \right)$ is positive definite in $\mathbb{R}^{n}\backslash \{0\}$. 
\end{enumerate} 
The dual norm of $F$ is defined as
\begin{equation}
    F^\circ(x):= \sup\limits_{\xi\neq 0} \dfrac{\langle x,\xi \rangle}{F(\xi)},\ \ x\in\mathbb{R}^{n},
    \label{F0}
\end{equation}
which is also a Minkowski norm. This definition implies the anisotropic Cauchy-Schwarz inequality:
\begin{equation}
    \langle x, \xi \rangle \leq F^\circ(x) F(\xi), \quad \forall\ x, \xi \in \mathbb{R}^{n},
    \label{eq:anisocauchy}
\end{equation}
with equality holding if and only if $\xi = F(\xi) DF^\circ(x)$ and $x = F^\circ(x) DF(\xi)$. Here $D$ denotes the gradient operator in $\mathbb{R}^{n}$.

The following properties relate $F$ and its dual $F^\circ$ (see e.g. \cite{Xia13,CS09}):
\begin{lemma}
    For any $x,\xi\in\mathbb{R}^{n}\backslash \{0\}$, we have
    \begin{align}
        \langle DF^\circ(x),x \rangle=F^\circ(x),&\ \ \langle DF(\xi),\xi \rangle=F(\xi). \label{fprop1}\\
        F\left( DF^\circ (x)\right)=1,&\ \ F^\circ \left( DF(\xi) \right) = 1, \label{fprop2}\\
       F^\circ(x) DF\left( DF^\circ(x) \right) = x,&\ \  F(\xi) DF^\circ \left( DF(\xi) \right) = \xi.\label{fprop3}
    \end{align}
    
\end{lemma}


Given a Minkowski norm $F$ in $\mathbb{R}^{n}$, the associated Wulff shape is defined as  
\begin{equation*}
    \mathcal{W}:=\left\lbrace x\in\mathbb{R}^{n}:F^\circ(x)<1 \right\rbrace.
\end{equation*}
For $x_0\in \mathbb{R}^{n}$ and $r>0$, we denote by $\mathcal{W}_r(x_0)=r\mathcal{W}+x_0$ the scaled and translated Wulff shape, which satisfies
\begin{equation*}
    \mathcal{W}_r(x_0)=\left\lbrace x\in\mathbb{R}^{n}:F^\circ(x-x_0)<r \right\rbrace.
\end{equation*}
When $F$ is the Euclidean norm, $\mathcal{W}_r(x_0)$ coincides with the Euclidean ball $\mathbb{B}_r(x_0)$ of radius $r$ centered at $x_0$. 

Let $\Omega\subset \mathbb{R}^n$ be an open bounded domain with smooth boundary $\partial\Omega$. The anisotropic perimeter $|\partial\Omega|_F$ of $\Omega$ is defined as 
\begin{equation*}
    |\partial\Omega|_F=\int_{\partial\Omega}F(\nu)d\mu,
\end{equation*}
where $\nu$ is the unit outward normal of $\partial\Omega$. 

\subsection{Weak inverse anisotropic mean curvature flow}

Given an open bounded domain $\Omega\subset\mathbb{R}^n$ with smooth mean-convex boundary $\partial\Omega$, the classical inverse anisotropic mean curvature flow (IAMCF) starting from $\partial\Omega$ describes a family of smooth closed  hypersurfaces $X: M \times [0,T) \to \mathbb{R}^{n}$ evolving according to
\begin{equation}
	\partial_t X(x,t) = \dfrac{1}{H_F}\nu_F(x,t), \label{IAMCF}
\end{equation}
where $H_F$ and $\nu_F=DF(\nu)$ denote the anisotropic mean curvature and the anisotropic unit normal vector of the evolving hypersurface $M_t=X(M,t)$, respectively. Suppose that the evolving hypersurface $M_t=\partial E_t$ are given by level sets of a function $u:\mathbb{R}^{n} \to \mathbb{R}$, with $E_t=\{x\in\mathbb{R}^{n}:u(x)<t\}$. If $u$ is smooth and $Du\neq 0$, then \eqref{IAMCF} is equivalent to the degenerate elliptic equation \begin{equation}\label{WIAMCF}
    \text{div}(DF(Du))=F(Du)
\end{equation}
in $\mathbb{R}^n\setminus \Omega$ and $\Omega=\{u<0\}$. 

The weak solution of \eqref{WIAMCF} was defined in \cite{DGX23} via the minimizing principle as in Huisken-Ilmanen \cite{HI01} for the isotropic case.  
Given an open bounded set $\Omega\subset \mathbb{R}^{n}$ with smooth boundary, a locally Lipschitz function $u$ on $\mathbb{R}^{n}$ is called a weak IAMCF starting from $\Omega$, if for every locally Lipschitz function $\varphi$  with $\{u\neq\varphi\}\subset\subset \mathbb{R}^{n}\setminus\Omega$ and any compact set $K\subset \mathbb{R}^{n}\setminus\Omega$ containing $\{u\neq\varphi\}$ we have 
  \begin{equation*}
        J_{F,u}^K(u) \leq J_{F,u}^K(\varphi)
    \end{equation*}
and $\Omega=\{u<0\}$, where the functional is defined as 
 \begin{equation*}
        J_{F,u}^K(\varphi) = \int_{K} \left[ F(D\varphi)+\varphi F(Du) \right] \mathrm{d}x.
    \end{equation*}
 We say that $u$ is a proper solution if in addition $\lim\limits_{|x|\to+\infty}u(x)=+\infty$.

The weak solution of \eqref{WIAMCF} can also be defined by set functional.
\begin{definition}
\label{def-minimizes}
    We say that $E$ minimizes $J_{F,u}$ in a set $U$ (minimizes on the outside, minimizes on the inside, resp.) if
    \begin{equation*}
        J_{F,u}^K(E) \leq J_{F,u}^K(G)
    \end{equation*}
    for any $G$ such that $E\triangle G\subset\subset U$  ($G\supset E$, $G\subset E$ resp.) and any compact set $K$ containing $E\triangle G$. Here $E\triangle G = (E\backslash G)\cup(G\backslash E)$, and 
    \begin{equation*}
        J_{F,u}^K(G) = \int_{\partial^* G \cap K} F(\nu)\mathrm{d}\mathcal{H}^{n-1} - \int_{G\cap K} F(Du) \mathrm{d}x
    \end{equation*}
    for a set $G$ of locally finite perimeter, and $\partial^*G$ denotes the reduced boundary of $G$.
\end{definition}

Then a locally Lipschitz function $u$ on $\mathbb{R}^{n}$ is a weak solution of IAMCF if for each $t>0$, $E_t=\{u<t\}$ minimizes $J_{F,u}$ in $\mathbb{R}^{n}\setminus\Omega$ in the sense of Definition \ref{def-minimizes}. Therefore, we also say that $E_t$ is a weak solution of IAMCF with initial data $E_0=\Omega$. 

We have the following comparison principle:
\begin{lemma}[see {\cite[Proposition 3.3]{DGX23}}]
    \label{lem-weakcom}
    Let $\{E_t\}_{t>0}$ and $\{F_t\}_{t>0}$ be two weak solutions of IAMCF with initial data $E_0$, $F_0$ respectively, and $E_0\subset F_0$, then $E_t\subset F_t$ as long as $E_t$ is precompact.
    
    In particular, for a given $E_0$ there exists at most one solution $\{E_t\}_{t>0}$ of \eqref{WIAMCF} such that $E_t$ is precompact. 
\end{lemma}



We state the following two properties of weak IAMCF, whose proofs closely follow those of Theorem 2.1 and Lemma 3.3 in \cite{HI01}, respectively. The minor modification involves replacing the Euclidean norm of $Du$ with $F(Du)$.
\begin{lemma}[Compactness]
    \label{lem-compact}
    Let $u_i$ be a sequence of solutions of \eqref{eq:WIAMCF} in open sets $U_i\subset\mathbb{R}^{n}$ such that
    \begin{equation*}
        u_i\to u,\ \ U_i\to U
    \end{equation*}
    locally uniformly, and for each $K\subset\subset U$,
    \begin{equation*}
        \sup\limits_{K} F(Du_i)\leq C(K)
    \end{equation*}
    for large $i$. Then $u$ is a solution of \eqref{eq:WIAMCF} in $U$.
\end{lemma}

\begin{lemma}[Smooth flow]
    \label{lem-smweak}
    Let $\{M_t\}_{c\leq t<d}$ be a smooth family of hypersurfaces with positive anisotropic mean curvature that solves \eqref{IAMCF} classically. Let $u=t$ on $M_t$, $u<c$ in the region bounded by $M_c$, and set $E_t=\{u<t\}$. Then for $c\leq t<d$, $E_t$ minimizes $J_{F,u}$ in $E_d\backslash \overline{E_c}$ in the sense of Definition \ref{def-minimizes}. In other words, smooth flows of IAMCF satisfy the weak formulation in the domain they foliate.
\end{lemma}
\begin{proof}
    For any $t\in (c,d)$, we need to show that: for any set $G$ with finite perimeter satisfying $E_t\triangle G \subset\subset E_d\backslash \overline{E_c} $ and any compact set $K$ containing $E_t\triangle G$, the following inequality holds:
    \begin{align}
        &\int_{\partial^* E_t \cap K} F(\nu)\mathrm{d}\mathcal{H}^{n-1} - \int_{E_t\cap K} F(Du) \mathrm{d}x \nonumber\\
        \leq &\int_{\partial^*G\cap K} F(\nu) \mathrm{d}\mathcal{H}^{n-1} - \int_{G\cap K} F(Du) \mathrm{d}x. \label{equ-jfu}
    \end{align}

    Consider the vector field $\nu_{F,u}=DF(Du)$. For $A=E_t\backslash G$ or $G\backslash E_t$, the equation \eqref{WIAMCF} implies
    \begin{equation}
        \int_{A} \mathrm{div}(\nu_{F,u}) \mathrm{d}x-\int_{A} F(Du) \mathrm{d}x=0
        \label{equ-A}
    \end{equation}
    Using the divergence theorem, property \eqref{fprop1}, and equation \eqref{equ-A}, we compute
    \begin{align}\label{s2.smoothflow}
        &\int_{\partial^* E_t \cap K} F(\nu) \mathrm{d}\mathcal{H}^{n-1}  - \int_{E_t \cap K} F(Du) \mathrm{d}x \nonumber\\
        =& \int_{\partial^* E_t \cap K} \nu_{\partial^* E_t}\cdot \nu_{F,u} \mathrm{d}\mathcal{H}^{n-1} - \int_{E_t\cap K} F(Du) \mathrm{d}x \notag\\
        =& \int_{E_t\cap K} \mathrm{div}(\nu_{F,u}) \mathrm{d}x - \int_{E_t\cap K} F(Du) \mathrm{d}x \notag\\
        =& \int_{G\cap K} \mathrm{div}(\nu_{F,u}) \mathrm{d}x - \int_{G\cap K} F(Du) \mathrm{d}x \notag\\
        =& \int_{\partial^* G\cap K} \nu_{\partial^* G}\cdot \nu_{F,u} \mathrm{d}\mathcal{H}^{n-1} - \int_{G\cap K} F(Du) \mathrm{d}x.
    \end{align}
    Applying anisotropic Cauchy-Schwarz inequality \eqref{eq:anisocauchy} and \eqref{fprop2} yields
    $$\nu_{\partial^* G}\cdot \nu_{F,u} \leq F(\nu_{\partial^* G}) F^\circ(\nu_{F,u})=F(\nu_{\partial^* G}).$$
    Substituting this inequality into the last line of \eqref{s2.smoothflow} establishes \eqref{equ-jfu}.
\end{proof}

\section{Local gradient estimate of anisotropic $p$-harmonic functions}
\label{sec:gradientest}

In this section, we prove the local gradient estimates for anisotropic $p$-harmonic functions as well as for the weak solution of IAMCF. 

Let $u$ solve the equation \eqref{s1.eq1} in the Wulff shape $\mathcal{W}_R(x_0)$. By setting $G(\xi)=\frac{1}{2}F^2(\xi)$, we rewrite \eqref{s1.eq1} as
\begin{equation}\label{finslerp2}
\Div\left( G^{\frac{p}{2}-1}(Du)G_{\xi}(Du) \right) = 2 G^{\frac{p}{2}}(Du).
\end{equation}
Let
\begin{equation*}
	A|_{\xi} = D^2 G(\xi) + \dfrac{p-2}{2}\dfrac{DG(\xi)\otimes DG(\xi)}{G(\xi)},
\end{equation*}
which is a uniformly positive definite matrix for $p>1$. Consider the linear elliptic operator:
\begin{equation}\label{eq:p-WIAMCF-2}
    \mathscr{L}(\varphi) = \Div \left( G^{\frac{p}{2}-1}(Du)A|_{Du}(D \varphi) \right) - p G^{\frac{p}{2}-1}(Du)\langle G_{\xi}(Du),D\varphi \rangle.
\end{equation}

In the following, we use the convention that
$$G_i=\dfrac{\partial G(\xi)}{\partial\xi_i},\ \ G_{ij}=\dfrac{\partial^2 G(\xi)}{\partial\xi_i\partial\xi_j},\ \ G_{ijk}=\dfrac{\partial^3 G(\xi)}{\partial\xi_i\partial\xi_j\partial\xi_k}.$$
We have the following Bochner type formula.
\begin{lemma}
    Let $f=G(Du)$. Then
    \begin{equation}\label{Bochner}
        \mathscr{L}(f) = f^{\frac{p}{2}-1}A_{kj}u_{j\ell}G_{\ell s}u_{sk},
    \end{equation}
    where $A_{ij}=(A|_{Du})_{ij}=G_{ij}(Du)+\frac{p-2}{2}\frac{G_{i}(Du)G_{j}(Du)}{f}.$
\end{lemma}
\begin{proof}
We first note that 
\begin{equation*}
    f_i=\frac{\partial f}{\partial x^i}=G_ku_{ki}.
\end{equation*}
Differentiating \eqref{finslerp2} with respect to $x^i$, we have
\begin{align}\label{dif}
 0=&\frac{\partial}{\partial x^i}\left(\frac{\partial}{\partial x^k}\left(f^{\frac{p}{2}-1}G_k\right)-2f^{\frac{p}{2}}\right)\nonumber\\
 =&\frac{\partial}{\partial x^k}\left(f^{\frac{p}{2}-1}G_{kl}u_{li}+\frac{p-2}{2}f^{\frac{p}{2}-2}f_iG_k\right)-pf^{\frac{p}{2}-1}f_i\nonumber\\
  =&\frac{\partial}{\partial x^k}\left(f^{\frac{p}{2}-1}G_{kl}u_{li}+\frac{p-2}{2}f^{\frac{p}{2}-2}G_lu_{li}G_k\right)-pf^{\frac{p}{2}-1}f_i\nonumber\\
 =&\dfrac{\partial}{\partial x^k} \left( f^{\frac{p}{2}-1} A_{kl} u_{li} \right) - p f^{\frac{p}{2}-1}f_{i}.
\end{align}
Hence, by \eqref{eq:p-WIAMCF-2} we have
\begin{align*}
	\mathscr{L}(f) =& \dfrac{\partial}{\partial x^k} \left( f^{\frac{p}{2}-1}A_{kj}f_{j} \right) - pf^{\frac{p}{2}-1}G_{i}f_{i} \notag\\
    =&\dfrac{\partial}{\partial x^k} \left( f^{\frac{p}{2}-1}A_{kj}u_{jl}G_l \right) - pf^{\frac{p}{2}-1}G_{i}f_{i}\nonumber\\
    =& \dfrac{\partial}{\partial x^k}\left( f^{\frac{p}{2}-1}A_{kj} u_{jl} \right) G_l + f^{\frac{p}{2}-1}A_{kj} u_{jl}G_{ls}u_{sk} - p f^{\frac{p}{2}-1}G_if_{i} \notag\\
	=& f^{\frac{p}{2}-1}A_{kj}u_{jl}G_{ls}u_{sk},
\end{align*}
as desired.
\end{proof}

To prove the local gradient estimate for $u$ in Theorem \ref{Sharp-C1}, we multiply $f$ by a suitable cut-off function $\eta\in C_c^\infty(\mathcal{W}_R(x_0))$ in the Wulff shape $\mathcal{W}_R(x_0)$. The required properties for $\eta$ will be determined later.  Let 
$$Q=f\eta$$
which achieves its maximum in the interior of $\mathcal{W}_R(x_0)$. We will derive a differential inequality satisfied by $Q$ at the maximum point. 

\begin{proof}[Proof of Theorem \ref{Sharp-C1}]
We first observe that, at the maximum point of $Q$ there hold
\begin{equation}\label{first-cond}
	0=D_i Q = f_i \eta + f \eta_i= G_ku_{ki}\eta+f\eta_i,
\end{equation}
\begin{equation}\label{second-cond}
	0 \geq \mathscr{L}(Q).
\end{equation}
Applying \eqref{first-cond} to \eqref{finslerp2}, we have
\begin{equation}\label{finslerlap3}
	G_{ik}u_{ik} - \dfrac{p-2}{2}\dfrac{G_k\eta_k}{\eta} = 2 f
\end{equation}
at the maximum point of $Q$. 

In the following, we calculate at the maximum point of $Q$ that
\begin{align*}
\mathscr{L}(Q) 
    =&\mathrm{div}\left( f^{\frac{p}{2}-1}A(DQ) \right) - p f^{\frac{p}{2}-1} \langle DG(Du), DQ \rangle \notag\\
	=& \eta \mathscr{L}(f) + \mathrm{div}\left( f^{\frac{p}{2}}A(D\eta) \right) \notag\\
	& + \langle f^{\frac{p}{2}-1} A (Df),D\eta \rangle - p f^{\frac{p}{2}} \langle DG(Du),D\eta \rangle \notag\\
	=& \eta\mathscr{L}(f) + f^{\frac{p}{2}}\mathrm{div}(A(D\eta)) \\
    &\quad - \left( 1+\dfrac{p}{2} \right)f^{\frac{p}{2}}\dfrac{\langle A(D\eta), D\eta\rangle}{\eta} - p f^{\frac{p}{2}}\langle DG(Du),D\eta \rangle,
    \end{align*}
where we used \eqref{first-cond} in the last equality. Note that
\begin{align}
	\mathrm{div}(A(D\eta))=&\dfrac{\partial}{\partial x^k}\left(A_{ik}\eta_i\right)\nonumber\\
    =&A_{ik}\eta_{ik}+\dfrac{\partial}{\partial x^k}\left(  G_{ki}+\dfrac{p-2}{2}\dfrac{G_kG_i}{G} \right)\eta_i  \notag\\
    =&A_{ik}\eta_{ik} + G_{kij}u_{jk}\eta_i+\frac{p-2}{2}\frac{G_{kl}u_{lk}G_i+G_kG_{il}u_{lk}}{G}\eta_i\nonumber\\
    &\quad -\frac{p-2}{2}\frac{G_kG_i}{G^2}f_k\eta_i\nonumber\\
    =& A_{ik}\eta_{ik} + G_{kij}u_{jk}\eta_i + \frac{p(p-2)}{4}\frac{(G_{k}\eta_{k})^2}{f\eta} \notag\\
    &+(p-2)G_{i}\eta_{i}-\frac{p-2}{2}\frac{G_{ik}\eta_{i}\eta_{k}}{\eta}, \label{eq:divaeta}
\end{align}
where we used \eqref{first-cond} and \eqref{finslerlap3} in the last equality. Hence by \eqref{Bochner} and \eqref{eq:divaeta}, we have
\begin{equation}\label{LQ}
    \begin{split}
       \mathscr{L}(Q)=&\eta f^{\frac{p}{2}-1}G_{kj}u_{j\ell}G_{\ell s}u_{sk}+f^{\frac{p}{2}}G_{kij}u_{jk}\eta_{i}\\
       &+f^{\frac{p}{2}}A_{ik}\eta_{ik}-\frac{p+2}{2}f^{\frac{p}{2}}\dfrac{G_{ij}\eta_{i}\eta_{j}}{\eta}\\
       &-\frac{p-2}{2}f^{\frac{p}{2}-1}\frac{(G_{k}\eta_{k})^2}{\eta}-2f^{\frac{p}{2}}G_{i}\eta_{i}.
    \end{split} 
\end{equation}


We estimate the terms on the right hand side of \eqref{LQ} separately. 

\textbf{Step 1. Estimate $G_{kij}u_{jk}\eta_{i}$. } As in \cite[Page 129]{CL07}, let $\{\phi_1,\phi_2,\cdots,\phi_n\}$ be the basis of $\mathbb{R}^n$ such that $G_{ij}=\delta_{ij}$ at the maximum point of $Q$ and $(G^{pq})$ the inverse matrix of $(G_{ij})$. Applying the Cauchy-Schwarz inequality for positive definite matrix $(G_{ij})$ yields
\begin{align*}
	G_{ijk}u_{ij}\eta_k =& G_{pk} G_{ijl}u_{ij}G^{lp}\eta_k\\
    \leq& \sqrt{ G_{ij}\eta_i\eta_j}\sqrt{ G^{pq}G_{ijp}u_{ij}G_{klq}u_{kl}} \\
    \leq&\dfrac{\varepsilon}{\eta}G_{ij}\eta_i\eta_j + \dfrac{\eta}{4\varepsilon} G^{pq}G_{ijp}u_{ij}G_{klq}u_{kl}.
\end{align*}
Since $F$ is homogeneous of degree one, there exists a positive constant $C_1=C_1(F)$ such that 
	\begin{equation}\label{small-third-der}
		FD^3 G|_{\xi}(\alpha,\beta,\gamma)\leq C_1\left( D^2G|_{\xi}(\alpha,\alpha) \cdot D^2G|_{\xi}(\beta,\beta) \cdot D^2G|_{\xi}(\gamma,\gamma) \right)^{\frac{1}{2}}
	\end{equation}
for any vectors $\alpha,\beta,\gamma\in\mathbb{R}^n$. This is due to that both sides of \eqref{small-third-der} are homogeneous of degree zero and $D^2G$ is positive definite. The estimate \eqref{small-third-der} then implies:
\begin{align*}
	&G^{pq}G_{ijp}u_{ij}G_{klq}u_{kl}\\
    = &\dfrac{1}{2f}FD^3G|_{Du}\left( G^{pq}\phi_p,u_{ij}\phi_i,\phi_j \right) \\
    &\qquad \times FD^3G|_{Du}\left( \phi_q,u_{kl}\phi_k,\phi_l \right) \notag\\
	\leq &\dfrac{C_1^2}{2f}\left( D^2G|_{Du}(G^{pq}\phi_p,G^{rq}\phi_r)D^2G|_{Du}(u_{ij}\phi_i,u_{sj}\phi_s)D^2G|_{Du}(\phi_j,\phi_j)\right)^{\frac{1}{2}}\notag\\
	&\ \ \ \ \ \ \ \ \times\left(D^2G|_{Du}(\phi_q,\phi_q)D^2G|_{Du}(u_{kl}\phi_k,u_{tl}\phi_t)D^2G|_{Du}(\phi_l,\phi_l) \right)^{\frac{1}{2}}\notag\\
	=&\dfrac{C_1^2\sqrt{n}}{2f} G_{ij}G_{kl}u_{ik}u_{jl}.
\end{align*}
This proves the bound:
\begin{align}\label{est:Gijk}
	\left|G_{ijk}u_{ij}\eta_k\right| \leq \dfrac{\varepsilon}{\eta}G_{ij}\eta_i\eta_j + \dfrac{\eta}{4\varepsilon}\dfrac{C_1^2\sqrt{n}}{2f} G_{ij}G_{kl}u_{ik}u_{jl},
\end{align}
where $\varepsilon>\sqrt{n}C_1^2/8$ is to be determined later. Substituting  \eqref{est:Gijk} into \eqref{LQ}  gives
\begin{equation}\label{LQ-1}
    \begin{split}
        \mathscr{L}(Q) \geq & \left(1-\dfrac{\sqrt{n}C_1^2 }{8\varepsilon }\right)\eta f^{\frac{p}{2}-1}G_{ij}G_{kl}u_{ik}u_{jl} + f^{\frac{p}{2}}A_{ik}\eta_{ik} \\
        &-\dfrac{p+2+2\varepsilon}{2}f^{\frac{p}{2}}\frac{G_{ij}\eta_{i}\eta_{j}}{\eta}-\frac{p-2}{2}f^{\frac{p}{2}-1}\frac{(G_{k}\eta_{k})^2}{\eta}-2f^{\frac{p}{2}}G_{i}\eta_{i}.
    \end{split}
\end{equation}

\textbf{Step 2. Estimate the term $G_{ij}G_{kl}u_{ik}u_{jl}$.} We choose a local orthonormal basis $\{\phi_1,\phi_2, \dots, \phi_n\}$ of $\mathbb{R}^{n}$ with respect to the metric $G_{ij}$ at the maximum point of $Q$, such that $\phi_1 = \dfrac{Du}{|Du|_G}$ and $G_{ij}=\delta_{ij}$. Then we have $u_1 = |Du|_G$ and $u_i = 0$ for $i \geq 2$. Since 
$$G_{ij}u_j=G_i\ \ \mbox{and}\ \ G_{ij}u_iu_j=2G$$ 
by 2-homogeneity of $G$, then $|Du|_G^2=G_{ij}u_iu_j = 2G$.  Consequently,
\begin{equation}
	f_k=G_{l}u_{lk}=G_{il}u_i u_{lk}=u_{ik}u_i = u_{1k}u_1,
    \label{s3.fk}
\end{equation} 
and \eqref{finslerlap3} becomes
\begin{equation*}
	\sum\limits_{i=1}^{n} u_{ii} - \dfrac{p-2}{2} \dfrac{G_k\eta_k}{\eta} = 2f.
\end{equation*}

Using the relation
\begin{equation*}
	\sum\limits_{j\geq 2} u_{jj} = 2 f + \dfrac{p-2}{2}\dfrac{G_k\eta_k}{\eta}-u_{11},
\end{equation*}
we derive by Cauchy-Schwarz inequality that
\begin{align}
	\sum\limits_{i,j=1}^{n} u_{ij}^2 \geq& u_{11}^2 + 2\sum\limits_{j=2}^{n}u_{j1}^2 + \sum\limits_{j\geq 2}u_{jj}^2 \notag\\
	\geq & u_{11}^2 + 2\sum\limits_{j=2}^{n} u_{j1}^2 + \dfrac{1}{n-1}\left( \sum\limits_{j\geq 2} u_{jj} \right)^2 \notag\\
	= & u_{11}^2 + 2\sum\limits_{j=2}^{n} u_{j1}^2 + \dfrac{1}{n-1}\left(2 f + \dfrac{p-2}{2}\dfrac{G_k\eta_k}{\eta}-u_{11} \right)^2 \notag\\
	= & \dfrac{4}{n-1}f^2 + \dfrac{n}{n-1}u_{11}^2 + 2\sum\limits_{j=2}^{n}u_{j1}^2 + \dfrac{(p-2)^2}{4(n-1)}\dfrac{(G_k\eta_k)^2}{\eta^2} \notag \\
	& + \dfrac{2(p-2)}{n-1}f\dfrac{G_k\eta_k}{\eta} - \dfrac{p-2}{n-1}\dfrac{G_k\eta_k}{\eta}u_{11} - \dfrac{4}{n-1}fu_{11}.\label{eq:aij1}
\end{align}
From \eqref{first-cond} and \eqref{s3.fk}, we deduce
\begin{align}
    u_{11}= &\dfrac{f_1}{u_1} = \frac{f_1u_1}{u_1^2}\nonumber\\
    =&\frac{f_iG_{ij}u_j}{2f}=\dfrac{f_iG_i}{2f}= -\dfrac{G_i\eta_i}{2\eta}
    \label{u11}
\end{align}
and
\begin{equation}
    \label{eq:fifj}
    \sum\limits_{j=1}^nu_{j1}^2 = \sum\limits_{j=1}^n \dfrac{f_k^2}{u_1^2} = \dfrac{G_{ij}f_if_j}{2f}=f\dfrac{G_{ij}\eta_i\eta_j}{2\eta^2}.
\end{equation}
Combining \eqref{eq:aij1}, \eqref{u11} and \eqref{eq:fifj}, we derive
\begin{align}
    \sum\limits_{i,j=1}^n u_{ij}^2\geq & \dfrac{4}{n-1}f^2 + \dfrac{n}{n-1} \dfrac{G_{ij}f_if_j}{2f} + \dfrac{(p-2)^2}{4(n-1)}\dfrac{(G_k\eta_k)^2}{\eta^2} \notag\\
	&+\dfrac{2(p-2)}{n-1}f\dfrac{G_k\eta_k}{\eta} + \dfrac{p-2}{2(n-1)}\dfrac{(G_k\eta_k)^2}{\eta^2} + \dfrac{2}{n-1}f\dfrac{G_k\eta_k}{\eta} \notag\\
	= & \dfrac{4}{n-1}f^2 + \dfrac{n}{n-1}f \dfrac{G_{ij}\eta_i \eta_j}{2\eta^2} +\dfrac{2(p-1)}{n-1}f\dfrac{G_k\eta_k}{\eta}\nonumber\\
    &\quad+ \dfrac{p(p-2)}{4(n-1)}\dfrac{(G_k\eta_k)^2}{\eta^2}.  \label{eq:aij}
\end{align}
Therefore, combining \eqref{LQ-1} and \eqref{eq:aij}, we obtain
\begin{align}\label{LQ-2}
        \mathscr{L}(Q)\geq&  \frac{4C_{\varepsilon}}{n-1}  \eta f^{\frac{p}{2}+1} + f^{\frac{p}{2}}A_{ik} \eta_{ik} \nonumber\\
        &+\left( -\frac{p+2+2\varepsilon}{2}+\frac{n}{2(n-1)}C_{\varepsilon} \right) f^{\frac{p}{2}}\frac{G_{ij}\eta_{i}\eta_{j}}{\eta} \nonumber\\
        &+\underbrace{\left( \frac{2(p-1)}{n-1}C_{\varepsilon}-2 \right)}_{(a)} f^{\frac{p}{2}}G_{i}\eta_{i} \nonumber\\
        &+\underbrace{\left(  -\frac{p-2}{2}+\frac{p(p-2)}{4(n-1)}C_{\varepsilon} \right)}_{(b)}f^{\frac{p}{2}-1}\frac{(G_{k}\eta_{k})^2}{\eta},
\end{align}
where we set $$C_{\varepsilon}=\left(1-\frac{\sqrt{n}C_{1}^2}{8\varepsilon}\right)>0.$$

Note that for $1<p\leq 2$, $0<C_{\varepsilon}<1$ and $n\geq 2$, there hold
\begin{align*}
    (a) <& \dfrac{2}{n-1}-2\leq 0,\\
    (b) =& (p-2)\left( \dfrac{pC_{\varepsilon}}{4(n-1)}-\dfrac{1}{2}  \right) \geq (p-2)\left( \dfrac{1}{2(n-1)}-\dfrac{1}{2} \right) \geq 0.
\end{align*}
Thus, the last term in \eqref{LQ-2} is nonnegative and can be thrown away. The fourth term in \eqref{LQ-2} can be merged into the first and third terms, by using the inequality
\begin{align*}
	G_k\eta_k =& G_{kl}u_l\eta_k \leq \sqrt{G_{kl}u_lu_k}\sqrt{G_{kl}\eta_l\eta_k} \notag\\
	=& \sqrt{2G} \sqrt{G_{kl}\eta_k\eta_l}\notag\\
	\leq & \dfrac{\eta f}{\mu} + \dfrac{\mu}{2\eta} G_{kl}\eta_{k}\eta_l,
\end{align*}
where $\mu>0$ is a constant which will be determined later. This yields
    \begin{align}
        \mathscr{L}(Q)\geq & \left[ \frac{4C_{\varepsilon}}{n-1}+\frac{2(p-1)C_{\varepsilon}}{(n-1)\mu}-\frac{2}{\mu}  \right] \eta f^{\frac{p}{2}+1}+f^{\frac{p}{2}}A_{ik}\eta_{ik} \notag \\
        &+ \left[ -\frac{p+2+2\varepsilon}{2}+\left( \frac{n+2\mu(p-1)}{2(n-1)}\right)C_{\varepsilon}-\mu \right] f^{\frac{p}{2}}\frac{G_{ij}\eta_{i}\eta_{j}}{\eta}.\label{LQ-3}
    \end{align}

\textbf{Step 3. Estimate the term $A_{ik}\eta_{ik}$}. We choose
\begin{equation*}
    \eta(x)=\theta\left( \dfrac{F^\circ(x-x_0)}{R} \right),
\end{equation*}
where $\theta(t)$ is a cut-off function satisfying $\theta(t)\equiv 1$ for $0\leq t\leq \frac{1}{2}$ and $\theta(t)\equiv 0$ for $t\geq 1$. Additionally, $\theta$ can be chosen such that its derivatives satisfy
\begin{equation}
    \dfrac{(\theta')^2}{\theta}\leq 40,\ \ \theta''\geq -40\theta\geq -40.\label{assum}
\end{equation}
A direct computation gives
 \begin{align}
 	\eta_i = & \theta'\left( \dfrac{F^\circ(x-x_0)}{R} \right)\dfrac{D_i F^\circ(x-x_0)}{R},\nonumber\\
 	\eta_{ij} = & \theta'\left( \dfrac{F^\circ(x-x_0)}{R} \right)\dfrac{D_j D_i F^\circ(x-x_0)}{R} \notag\\
    &+  \theta''\left( \dfrac{F^\circ(x-x_0)}{R} \right) \dfrac{D_i F^\circ(x-x_0)}{R} \dfrac{D_j F^\circ(x-x_0)}{R}.\label{eq:etaij}
 \end{align}
Since $\theta'=\theta''=0$ if $F^\circ(x-x_0)\leq {R}/{2}$, we restrict our consideration to $F^\circ(x-x_0)\geq {R}/{2}$. Otherwise $A_{ik}\eta_{ik}=0$ according to the expression \eqref{eq:etaij}.

Let $G^\circ(\zeta)=\dfrac{1}{2}(F^\circ)^2(\zeta)$. Note that
\begin{align}
	&\left( D_i D_j F^\circ(x-x_0)\right)\nonumber \\
    \leq & \left( \dfrac{D_i D_j F^\circ(x-x_0) F^\circ(x-x_0) + D_i F^\circ(x-x_0) D_j F^\circ(x-x_0)}{F^\circ(x-x_0)} \right)\notag \\ =& \left( \dfrac{D_iD_j G^{\circ}(x-x_0)}{F^\circ(x-x_0)}\right),
	\label{eq:ineq1}
\end{align}
we then obtain
\begin{equation*}
	A_{ik}D_iD_k F^\circ(x-x_0) \leq \dfrac{2}{R}A_{ik}D_iD_k G^{\circ}(x-x_0).
\end{equation*}
Furthermore, by the assumption \eqref{assum} on $\theta$ and $(A_{ij})\leq (G_{ij})$ for $p\leq 2$, we have the inequality
\begin{align}
	A_{ik}\eta_{ik} \geq& -\dfrac{80}{R^2} A_{ik}D_iD_k G^{\circ}(x-x_0) \nonumber\\
    &\qquad - \dfrac{40}{R^2}A_{ik}D_i F^\circ(x-x_0) D_k F^\circ(x-x_0) \notag\\
    \geq&-\dfrac{80}{R^2}G_{ik}D_iD_k G^{\circ}(x-x_0)\nonumber\\
    &\qquad -\dfrac{40}{R^2}G_{ik}D_i F^\circ(x-x_0) D_k F^\circ(x-x_0)\nonumber\\
    \geq &-\frac{C(n,F)}{R^2}.
	\label{eq:aikeik}
\end{align}

\textbf{Step 4. Complete the proof}. Combining \eqref{LQ-3} and \eqref{eq:aikeik}, we obtain that at the maximum point of $Q$,
\begin{align}\label{LQ-4}
	0\geq &\mathscr{L}(Q) \nonumber\\
    \geq& \underbrace{\left[   \frac{4C_{\varepsilon}}{n-1}+\frac{2(p-1)C_{\varepsilon}}{(n-1)\mu}-\frac{2}{\mu}  \right]}_{(c)}f^{\frac{p}{2}+1}\eta -\frac{C(n,F)}{R^2}f^{\frac{p}{2}}\notag\\
	&+\underbrace{\left[ -\frac{p+2+2\varepsilon}{2}+\left( \frac{n+2\mu(p-1)}{2(n-1)}\right)C_{\varepsilon}-\mu \right]}_{(d)}f^{\frac{p}{2}}\dfrac{G_{ij}\eta_i\eta_j}{\eta}.
\end{align}
Now we take $\varepsilon=\sqrt{n}C_1^2$ and $\mu=n-1$. Then $C_{\varepsilon}=\dfrac{7}{8}$, and we compute
\begin{align*}
(c) = &\dfrac{3}{2(n-1)}+\dfrac{7(p-1)}{4(n-1)^2},\\
(d) = & -\sqrt{n}C_1^2-\dfrac{2n+1}{2}+\dfrac{3}{8}(p-1)+\dfrac{7n}{16(n-1)}.
\end{align*}
Since $p\geq 1$, we deduce from \eqref{LQ-4} that
\begin{equation}\label{LQ-5}
    0 \geq \dfrac{1}{2(n-1)}Q -\left( \sqrt{n}C_1^2+\dfrac{2n+1}{2} \right)\dfrac{G_{ij}\eta_i\eta_j}{\eta}-\dfrac{C(n,F)}{R^2}.
\end{equation}
On the other hand, from \eqref{assum} we have the bound
\begin{equation}\label{LQ-6}
    \frac{G_{ij}\eta_i\eta_j}{\eta}\leq \frac{C(n,F)}{R^2}.
\end{equation}
Substituting \eqref{LQ-6} into \eqref{LQ-5} implies
\begin{equation*}
	Q \leq \dfrac{C(n,F)}{R^2}.
\end{equation*}
Since $\eta\equiv 1$ on the set where $F^\circ(x-x_0)\leq {R}/{2}$, the desired result \eqref{equ:sharpgradient} follows.
\end{proof}

As a corollary, we obtain the local gradient estimate for the weak solution of IAMCF.

\begin{proof}[Proof of Corollary \ref{s1.cor1}]
    Since $u$ is Lipschitz and $\Omega$ is bounded, the estimate \eqref{s1.eq:sharpgrad} holds automatically on any compact set containing $\Omega$ due to the boundedness of $F(Du)$ and the fact that $F^\circ(x)$ is bounded away from zero. Therefore, it suffices to prove \eqref{s1.eq:sharpgrad} outside a sufficiently large compact set containing $\Omega$. 
    
    Fix $R > 0$ such that $\Omega \subset \mathcal{W}_R$. Now consider any $x \in \mathbb{R}^n \setminus \overline{\mathcal{W}_{4R}}$, which implies $F^\circ(x) > 4R$. We first show that the set $\mathcal{W}_{F^\circ(x) - R}(x)$ is contained in $\mathbb{R}^n \setminus \overline{\mathcal{W}_{R}}$. Indeed, for any $y\in \mathcal{W}_{F^\circ(x)-R}(x)$, we have $F^\circ(x-y) < F^\circ(x)-R$. By the triangle inequality for the norm $F^\circ$,
    $$F^\circ(y) \geq F^\circ(x)-F^\circ(x-y) > F^\circ(x)-(F^\circ(x)-R)=R,$$
    which implies $y \notin \overline{\mathcal{W}_R}$. Hence, $\mathcal{W}_{F^\circ(x) - R}(x) \subset \mathbb{R}^n \setminus \overline{\mathcal{W}_R}$.

    Taking the limit $p \to 1^+$ in the local gradient estimate \eqref{equ:sharpgradient} and applying it to $u$ in $\mathcal{W}_{(F^\circ(x) - R)/2}(x)$ yields
	\begin{equation*}
		F(Du)(x) \leq \dfrac{C(n, F)}{F^\circ(x) - R}.
	\end{equation*}
    Since $F^\circ(x) > 4R$, we have $F^\circ(x) - R \geq \frac{1}{2}F^\circ(x)$. Thus,
	\begin{equation*}
		F(Du)(x) \leq \dfrac{2C(n, F)}{F^\circ(x)}.
	\end{equation*}
	This combined with the bound of $F(Du)$ on the compact set $\overline{\mathcal{W}_{4R}}\setminus\Omega$ implies the estimate \eqref{s1.eq:sharpgrad}. 
\end{proof}

\section{Asymptotic behaviour of weak IAMCF}\label{sec:main}
    With the gradient estimate from Corollary \ref{s1.cor1} in hand, we now prove Theorem \ref{thm:anisotropic-asymptotic} on the asymptotic behavior of the weak IAMCF.  Our proof uses a similar argument as in \cite{HI01} and \cite{BFM24}, and proceeds in three steps.
	
    \textbf{Step 1.} Let $u$ be the weak solution of IAMCF \eqref{eq:WIAMCF} starting from $\Omega$. We show that we can ``blow down'' $u$  to a weak IMCF on $\mathbb{R}^n\setminus\{0\}$. 
    
    For each $s \geq 1$, define the rescaled functions 
    \begin{equation}\label{s4.us}
        u_s (x)=u(sx)-(n-1)\log s,\ \ \forall\ x\in \mathbb{R}^n\backslash\Omega_s,
    \end{equation}
    where the set $\Omega_s=\{\frac{x}{s}:x\in\Omega\}$.  The family $(u_s)_{s \geq 1}$ consists of weak solutions of IAMCF on $\mathbb{R}^n \setminus \Omega_s$. 
    
    By the comparison estimate \eqref{C0 estimates} and gradient estimate \eqref{s1.eq:sharpgrad}, this family $(u_s)_{s\geq 1}$ is uniformly bounded and equi-Lipschitz. In fact, for any $x \in \mathbb{R}^n \setminus \Omega_s$, we have $sx \in \mathbb{R}^n \setminus \Omega$. The estimate \eqref{C0 estimates} implies that 
    \begin{align*}
        u_s(x) &= u(sx) - (n - 1)\log s \\
        &\leq (n - 1)\log F^\circ(sx) -(n-1)\log r_1 - (n - 1)\log s \\
        &= (n - 1)\log F^\circ(x) -(n-1)\log r_1,
    \end{align*}
where we used the 1-homogeneity of $F^\circ$. Similarly, we have 
\begin{align*}
        u_s(x) &\geq  (n - 1)\log F^\circ(x) -(n-1)\log r_2.
    \end{align*}
Here $0<r_1<r_2$ are two constants such that $\mathcal{W}_{r_1}\subset\Omega\subset \mathcal{W}_{r_2}$.  By the gradient estimate \eqref{s1.eq:sharpgrad} and 1-homogeneity of $F$ and $F^\circ$, we also have 
    \begin{equation}\label{s4.pf2}
        F(Du_s(x)) = s F(Du(sx)) \leq s \cdot \frac{C(n, F, \Omega)}{F^\circ(sx)} = \frac{C(n, F, \Omega)}{F^\circ(x)}
    \end{equation}
for any $x\in \mathbb{R}^n\setminus\Omega_s$.  
    
    Therefore, the Arzel\`a-Ascoli theorem implies precompactness of $(u_s)_{s\geq 1}$ with respect to local uniform convergence on $\mathbb{R}^n \setminus \{0\}$. Furthermore, the compactness theorem (see Lemma \ref{lem-compact}) ensures that any limit point $v(x)=\lim\limits_{s_k\to\infty} u_{s_k}(x)$ solves the weak IAMCF \eqref{eq:WIMCF} on $\mathbb{R}^n \setminus \{0\}$, and  satisfies the same estimates
    \begin{align}\label{s4.pf3}
    (n-1)\log F^\circ(x) - (n-1)\log r_2\leq v(x) \leq (n-1)\log F^\circ(x) - (n-1)\log r_1,
    \end{align}
 and
    \begin{equation}\label{s4.pf4}
    F(Dv)(x)\leq \dfrac{C(n,F,\Omega)}{F^\circ(x)}
    \end{equation}
  for all $x\in \mathbb{R}^n\backslash \{0\}$.
	
    {\bf Step 2.} We prove that every blow down limit $v$ of $(u_s)_{s\geq 1}$ is an expanding Wulff shape and has the form
    \begin{equation}\label{s4.pf5}
        v(x)=(n-1)\log F^\circ(x)+\gamma
    \end{equation}
    on $\mathbb{R}^n\backslash\{0\}$ for some $\gamma\in\mathbb{R}$. 
    
    Define the eccentricity function $\theta_v:\mathbb{R}\to\mathbb{R}$ by
    $$\theta_v(t)=\dfrac{R_v(t)}{r_v(t)}$$
    for every $t\in\mathbb{R}$, where $\{x\in\mathbb{R}^n:r_v(t)\leq F^\circ(x)\leq R_v(t)\}$ is the smallest anisotropic annulus containing $\{v=t\}$. Note that the level set $\{v = t\}$ is compact. This follows from the $C^0$ estimate \eqref{s4.pf3}, which implies that $\{v = t\}$ is contained in a bounded annular region $\left\{ r_1\mathrm{e}^{\frac{t}{n-1}} \leq F^\circ(x) \leq r_2 \mathrm{e}^{\frac{t}{n-1}} \right\}$ away from the origin. Consequently, $\theta_v(t)$ is well-defined and finite for all $t$. Clearly, $\theta_v(t)\geq 1$ and $\theta_v(t)=1$ if and only if $\{v=t\}$ is the boundary of a Wulff shape. We next show that $\theta_v(t)\equiv 1$. 
    
    Firstly, we prove that $t\mapsto \theta_v(t)$ is non-increasing. Indeed, for any $\tilde{t}>t$ and any $x\in\{v=\tilde{t}\}$, the comparison principle (Lemma \ref{lem-weakcom}) implies
    \begin{equation*}
    	r_v(t)\mathrm{e}^{\frac{\tilde{t}-t}{n-1}} \leq F^\circ (x) \leq R_v(t)\mathrm{e}^{\frac{\tilde{t}-t}{n-1}}.
    \end{equation*}
    Taking suprema and infima yields
    $$R_v(\tilde{t})=\sup\limits_{x\in\{v=\tilde{t}\}}F^\circ (x)\leq R_v(t)\mathrm{e}^{\frac{\tilde{t}-t}{n-1}},$$
    and
    $$r_v(\tilde{t})=\inf\limits_{x\in\{v=\tilde{t}\}} F^\circ (x) \geq r_v(t)\mathrm{e}^{\frac{\tilde{t}-t}{n-1}}.$$
    Hence,
    $$\theta_v(\tilde{t})=\dfrac{R_v(\tilde{t})}{r_v(\tilde{t})}\leq \dfrac{R_v(t)\mathrm{e}^{\frac{\tilde{t}-t}{n-1}}}{r_v(t)\mathrm{e}^{\frac{\tilde{t}-t}{n-1}}}=\dfrac{R_v(t)}{r_v(t)}=\theta_v(t)$$
    for any $\tilde{t}>t$. 
    
    Secondly, on $\mathbb{R}^n\setminus\{0\}$ it makes sense to ``blow up'' $v$ to another weak IMCF $h$ on $\mathbb{R}^n\setminus\{0\}$.  We can show  that $\theta_h(t)$ is constant and is equal to $\theta_h(0)=\sup_t \theta_v(t)\in [1,+\infty)$. In fact, for any $0<s\leq 1$, consider the rescaled functions
    $$v_s(x)=v(sx)-(n-1)\log s,\ \ x\in\mathbb{R}^n\backslash\{0\}.$$
    Then $v_s$ satisfies the same estimates as in \eqref{s4.pf3} and \eqref{s4.pf4}.  By Arzel\`a-Ascoli theorem, there exists a function $h:\mathbb{R}^n\backslash\{0\}\to\mathbb{R}$ such that $(v_s)_{s\leq 1}$ locally uniformly converges to $h$ as $s\to 0^+$.
    Since 
    \begin{align*}
    	\sup\limits_{x\in\{v_s=t\}} F^\circ(x) =& \sup\limits_{y\in\{v=(n-1)\log s+t\}} F^\circ\left( \frac{y}{s} \right) \\
    	=& \sup\limits_{y\in\{v=(n-1)\log s+t\}} \frac{1}{s}F^\circ\left( y \right) \\
        =& \dfrac{1}{s}R_v((n-1)\log s+t)
    \end{align*}
    and
    \begin{align*}
    	\inf\limits_{x\in\{v_s=t\}} F^\circ(x) =& \inf\limits_{y\in\{v=(n-1)\log s+t\}} F^\circ\left( \frac{y}{s} \right) \\
    	=& \dfrac{1}{s}r_v((n-1)\log s+t),
    \end{align*}
    we have
    $$\theta_{v_s}(t) =\dfrac{R_v((n-1)\log s + t)}{r_v((n-1)\log s + t)} = \theta_v ((n-1)\log s+t).$$
    Therefore,
    $$\theta_h(t) = \lim\limits_{s\to 0^+} \theta_{v_s}(t) =\sup\limits_{-\infty<s<+\infty} \theta_v(s),$$
    which is finite.
    
    Lastly, we prove that
    $$\theta_h(0)=1.$$ Suppose by contradiction that $\theta_h(0)>1$, then the level set $\{h=0\}\subset \{x\in\mathbb{R}^n:r_h(0)\leq F^\circ(x) \leq \theta_h(0)r_h(0)\}$, and it touches both the Wulff shapes $\{F^\circ(x)=r_h(0)\}$ and $\{F^\circ(x)=\theta_h(0)r_h(0)\}$ but differs from both boundary Wulff shapes. We perturb $\{F^\circ(x)\leq r_h(0)\}$ outward and $\{F^\circ(x)\leq \theta_h(0)r_h(0)\}$ inward along their anisotropic normals of the boundary to obtain two domains $U^-$ and $U^+$ respectively satisfying:
    \begin{itemize}
        \item $\{F^\circ(x)\leq r_h(0)\}\subset U^- \subset \{h\leq 0\}$ and $\{h\leq 0\}\subset U^+ \subset \{F^\circ(x)\leq \theta_h(0)r_h(0)\}$;
        \item $U^-$ and $U^+$ are star-shaped with smooth strictly $F$-mean convex boundaries (i.e., the anisotropic mean curvature of the boundary is positive).
    \end{itemize}
    Starting from $\partial U^-$ and $\partial U^+$, the smooth solutions  of inverse anisotropic mean curvature flow exist for all time (see \cite[Theorem 1.1]{Xia17}) and coincide with the corresponding weak solutions (see Lemma \ref{lem-smweak}). Denote by $(U_t^-)_{t\geq 0}$ and $(U_t^+)_{t\geq 0}$ the two weak (and smooth) IAMCF starting at $U^-$ and $U^+$, respectively. By the strong maximum principle for smooth flows, we have
    \begin{equation*}
        F^\circ(x) > r_h(0)\mathrm{e}^{\frac{t}{n-1}}\quad  \mathrm{for}\ x\in\partial U_t^-
    \end{equation*}
and 
     \begin{equation*}
   F^\circ(x) < \theta_h(0)r_h(0)\mathrm{e}^{\frac{t}{n-1}}\ \quad \mathrm{for}\ x\in\partial U_t^{+}.
    \end{equation*}
    The weak comparison principle (see Lemma \ref{lem-weakcom}) also yields the containment
    \begin{equation*}
        U_t^- \subset \{h\leq t\} \subset U_t^+.
    \end{equation*}
    Then it follows that $\theta_h(t)<\theta_h(0)$, contradicting the constancy of $\theta_h$. 
    
    In conclusion $\theta_h(t)\equiv 1$. Since $1\leq \theta_v (t) \leq \theta_h(t)=1$, we conclude that $\theta_v(t)\equiv 1$ and thus $v$ has the form \eqref{s4.pf5}.

    {\bf Step 3.} We prove that the constant $\gamma$ in \eqref{s4.pf5} is given by
    \begin{equation}\label{s4.gamma}
        \gamma = \log\left( \dfrac{|\partial \mathcal{W}|_F}{|\partial\Omega^*|_F} \right).
    \end{equation}
    For the weak solution $u$ of IAMCF starting from $\Omega$, we consider the auxiliary function $$w=\mathrm{e}^{\frac{u-\gamma}{n-1}}.$$ 
    Let  $v$ be a limit  of the rescaled solution $(u_s)_{s\geq 1}$ in \eqref{s4.us}. There exists a subsequence  $s_k\to +\infty$ such that
    $$u_{s_k} \to v = (n-1)\log F^\circ(x)+\gamma$$
    locally uniformly. Hence, for every $\varepsilon>0$, there exists $k_{\varepsilon}\in\mathbb{N}$ such that
    \begin{equation}\label{s4.inclu}
        \left\{F^\circ(x) \leq \dfrac{s_k}{1+\varepsilon} \right\} \subset \{ w\leq s_k \} \subset \left\{ F^\circ(x) \leq \dfrac{s_k}{1-\varepsilon} \right\},\ \ \forall\ k\geq k_{\varepsilon}.
    \end{equation}
    Since the Wulff shape and $\text{int}\{u\leq t\}$ are both strictly outward $F$-minimizing, the inclusion relationship of sets \eqref{s4.inclu} implies the following inequality for the anisotropic perimeters
       $$\left|\left\{ F^\circ(x) = \dfrac{s_k}{1+\varepsilon}\right\} \right|_F \leq |\partial\{w\leq s_k\}|_F \leq \left|\left\{ F^\circ(x)=\dfrac{s_k}{1-\varepsilon} \right\} \right|_F.$$
    By the exponential growth of the anisotropic perimeter under weak IAMCF (\cite[Proposition 3.4]{DGX23}), we have 
    \begin{align*}
        |\partial\{w\leq s_k\}|_F = &|\partial\{u\leq (n-1)\log s_k +\gamma\}|_F\\
        =&\mathrm{e}^{\gamma}s_k^{n-1}|\partial\Omega^*|_F,
    \end{align*}
    where $\Omega^*$ is the strictly outward $F$-minimizing hull of $\Omega$. We also have $|\partial\mathcal{W}_r|_F = r^{n-1} |\partial\mathcal{W}|_F$ for the anisotropic perimeter of the Wulff shape. Therefore, 
    \begin{equation*}
    	\left( \frac{s_k}{1+\varepsilon} \right)^{n-1} |\partial\mathcal{W}|_F 
    	\leq \mathrm{e}^{\gamma} s_k^{n-1} |\partial\Omega^*|_F 
    	\leq \left( \frac{s_k}{1-\varepsilon} \right)^{n-1} |\partial\mathcal{W}|_F.
    \end{equation*}
    Dividing both sides by $s_k^{n-1} |\partial\mathcal{W}|_F > 0$ yields
    \begin{equation*}
    	(1+\varepsilon)^{1-n} 
    	\leq \mathrm{e}^{\gamma} \frac{|\partial\Omega^*|_F}{|\partial\mathcal{W}|_F} 
    	\leq (1-\varepsilon)^{1-n}.
    \end{equation*}
    The result \eqref{s4.gamma} follows by taking $\varepsilon \to 0^+$.

    Finally, for $x$ with $F^\circ(x)\to\infty$, 
    \begin{align*}
        u(x)-(n-1)\log F^\circ(x)=&u\left(F^\circ(x)\cdot \frac{x}{F^\circ(x)}\right)-(n-1)\log F^\circ(x)\\
        \to&(n-1)\log F^\circ\left(\frac{x}{F^\circ(x)}\right)+\gamma\\
        =&\log\left( \dfrac{|\partial \mathcal{W}|_F}{|\partial\Omega^*|_F} \right).
    \end{align*}
    This gives the asymptotical behavior \eqref{s1.ulim} of $u$ in Theorem \ref{thm:anisotropic-asymptotic}. The proof in Step 2 implies that the expanding Wulff shape is the only solution to \eqref{eq:WIAMCF} on $\mathbb{R}^n\setminus\{0\}$ with compact level sets.

\end{document}